\documentclass[12pt]{amsart}
\usepackage{amscd}      
\usepackage{amssymb}
\usepackage{amsmath, amsthm, graphics}
\usepackage{xypic}      
\LaTeXdiagrams          
\usepackage[all,v2]{xy}
\xyoption{2cell} \UseAllTwocells \xyoption{frame} \CompileMatrices
\allowdisplaybreaks[3]

\usepackage{latexsym}
\usepackage{epsfig}
\usepackage{amsfonts}
\usepackage{enumerate}
\usepackage{times}

\newtheorem{theorem}[subsection]{Theorem}
\newtheorem{corollary}[subsection]{Corollary}
\newtheorem{lemma}[subsection]{Lemma}
\newtheorem{assump}[subsection]{Assumption}

\newtheorem{proposition}[subsection]{Proposition}
\theoremstyle{definition}

\theoremstyle{remark}

\theoremstyle{remark}

\newtheorem{remark}[subsection]{Remark}

\numberwithin{equation}{section}

\newcommand{\com}{\mathbb{C}}

\newcommand{\tpi}{\tilde{\pi}}

\newcommand{\X}{\mathcal{X}}
\newcommand{\Y}{\mathcal{Y}}
\newcommand{\Z}{\mathcal{Z}}

\newcommand{\V}{\mathcal{V}}

\newcommand{\sI}{\mathcal{I}}

\newcommand{\bone}{\mathbf{1}}
\newcommand{\pbar}{\bar{p}}

\def\<{\left\langle}
\def\>{\right\rangle}

\begin{document}

\title{Chern classes of Deligne-Mumford stacks and their coarse moduli spaces}
\author{Hsian-Hua Tseng}
\address{Department of Mathematics\\ University of Wisconsin-Madison\\ Van Vleck Hall, 480 Lincoln Drive\\ Madison, WI 53706-1388\\ USA}
\email{tseng@math.wisc.edu}

\date{\today}

\begin{abstract}
Let $X$ be a complex projective algebraic variety with Gorenstein quotient singularities and $\X$ a smooth Deligne-Mumford stack having $X$ as its coarse moduli space. We show that the CSM class $c^{SM}(X)$ coincides with the pushforward to $X$ of the total Chern class $c(T_{I\X})$ of the inertia stack $I\X$. We also show that the stringy Chern class $c_{str}(X)$ of $X$, whenever is defined, coincides with the pushforward to $X$ of the total Chern class $c(T_{II\X})$ of the double inertia stack $II\X$. Some consequences concerning stringy/orbifold Hodge numbers are deduced.

\end{abstract}

\maketitle

\section{Introduction}
Let $X$ be a complex projective algebraic variety with Gorenstein quotient singularities. In attempt to associate invariants to $X$, there are at least two possible approaches: one can either view $X$ as a {\em singular} variety by itself or view $X$ as the coarse moduli space of a {\em smooth} Deligne-Mumford stack $\X$. Viewing as a singular variety we have the CSM class $c^{SM}(X)$ naturally associated to $X$. An important property of CSM class is that its degree is equal to the topological Euler characteristic of $X$:
$$\chi(X)=\int_X c^{SM}(X).$$
The starting point of this note is the observation that $\chi(X)$ is equal to the Euler characteristic $\chi(I\X)$ of the {\em inertia stack} $I\X$. By Gauss-Bonnet formula for Deligne-Mumford stacks we know that $$\chi(I\X)=\int_{I\X}c_{top}(T_{I\X}),$$ where $c_{top}(T_{I\X})$ is the top Chern class of the tangent bundle $T_{I\X}$ of the inertia stack $I\X$. Therefore the degrees of $c^{SM}(X)$ and $c(T_{I\X})$ are the same. The first result of this note, Theorem \ref{csm_class}, says that this equality in fact holds for classes themselves: the pushforward to $X$ of $c(T_{I\X})$ is equal to $c^{SM}(X)$. This is a simple consequence of a comparison of the characteristic functions $\bone_X, \bone_{I\X}$ and MacPherson's natural transformation for Deligne-Mumford stacks \cite{o}.

Under some natural assumption on the singularities of $X$, such as being log terminal, one can define (\cite{dflnu}, \cite{alu}) the {\em stringy Chern class} $c_{str}(X)$ for $X$. In view of the above, it is reasonable to hope that $c_{str}(X)$ can be expressed using some kind of Chern class for the stack $\X$. We show (Theorem \ref{stringy_class}) that $c_{str}(X)$ is equal to the pushforward to $X$ of the total Chern class of the tangent bundle of the {\em double inertia stack} $II\X$ of $\X$. This implies some formulas for stringy Hodge numbers.

It is known (see e.g. \cite{ya}) that if $\X$ and $X$ are {\em K-equivalent}, i.e. the natural map $\pi:\X\to X$ is birational and we have $K_\X=\pi^*K_X$, then stringy Hodge numbers of $X$ coincide with orbifold Hodge numbers of $\X$. Together with the above results we find some formulas for orbifold Hodge numbers. In particular we prove in Proposition \ref{orb_numbers} a conjecture in \cite{jt}.

\section*{Acknowledgment}
The author thanks P. Aluffi and K. Behrend for helpful discussions, and referees for useful suggestions. This work is done during a stay at Institut Mittag-Leffler (Djursholm, Sweden). It is a pleasure to acknowledge their hospitality and support. In addition the author is supported in part by NSF grant DMS-0757722.

\section{Preliminaries}
We work over $\com$. Throughout the paper we will make the following assumption.

\begin{assump}\label{1st_assumption}
$\X$ is a smooth separated Deligne-Mumford stack of finite type over $\com$. Its coarse moduli space $X$ is a projective variety of finite type over $\com$. The structure map is denoted by $\pi:\X\to X$.
\end{assump}
From the scheme theory perspective, Assumption \ref{1st_assumption} means that $X$ is a projective variety of finite type with quotient singularities, and $\X$ is a smooth separated Deligne-Mumford stack having the (singular) variety $X$ as its coarse moduli space.

For an $\X$ as in Assumption \ref{1st_assumption} let $T_\X$ be the tangent bundle of $\X$. By definition $T_\X$ is a vector bundle over $\X$. As a locally free sheaf, $T_\X$ is defined to be the dual of the sheaf $\Omega^1_\X$ of differentials on $\X$. See e.g. \cite{vis}, 7.20 (ii) for the definition of $\Omega^1_\X$. The paper \cite{vis} also constructs the theory of Chow groups (with rational coefficients) for Deligne-Mumford stacks. In particular the theory of Chern classes is constructed there. Given a vector bundle $\V$ over $\X$, we have the total Chern class of $\V$ which we denote by $c(\V)$. The class $c(\V)$ belongs to $A^*(\X)_\mathbb{Q}$, the Chow group of $\X$ with $\mathbb{Q}$-coefficients. In particular, we write $c(T_\X)\in A^*(\X)_\mathbb{Q}$ for the total Chern class of the tangent bundle $T_\X$.

Consider a Deligne-Mumford stack of the form\footnote{Stacks of this form are called quotient stacks.} $\X=[U/G]$ where $U$ is a smooth scheme and $G$ is a linear algebraic group. The paper \cite{eg} constructs a theory of equivariant Chow groups (with integer coefficients) for the $G$-action on $U$, denoted by $A^G_*(U)$. By \cite{eg}, Proposition 14, we have $$A^G_*(U)\otimes \mathbb{Q}=A^{\text{dim} U-*}(\X)_\mathbb{Q}.$$
The tangent bundle $T_U$ is a $G$-equivariant vector bundle on $U$. The construction of \cite{eg}, Section 2.4 associates to $T_U$ its equivariant total Chern classes $c^G(T_U)\in A_*^G(U)\otimes \mathbb{Q}$. Under the above identification of Chow groups, we have $c^G(T_U)=c(T_\X)$. 

\begin{remark} 
We may consider the equivariant Chern class $c^G(T_U)$ as a class in the equivariant cohomology $H^*_G(U)\otimes \mathbb{Q}$ by using the cycle map.
\end{remark}

In the paper we make heavy use of the theory of constructible functions on Deligne-Mumford stacks. Our reference for this is \cite{j}, to which we refer the readers for a detailed treatment of of this. Below we recall some aspects of the theory.

Let $\X$ be a Deligne-Mumford stack as in Assumption \ref{1st_assumption}. Denote by $\X(\com)$ the set of $\com$-points of $\X$. By \cite{j}, Definition 4.1, a subset of $\X(\com)$ is {\em constructible} if it is a finite union of sets of the form $\X_i(\com)$ where each $\X_i$ is a finite type substack of $\X$. A function $\phi: \X(\com)\to \mathbb{Q}$ is called {\em constructible} if $\phi(\X(\com))$ is finite and $\phi^{-1}(c)\subset \X(\com)$ is constructible for any $c\in \phi(\X(\com))\setminus \{0\}$, see \cite{j}, Definition 4.3. Denote by $CF(\X)$ the group of constructible functions on $\X$. For a constructible set $C\subset \X(\com)$ define its {\em characteristic function} $\bone_C:\X(\com)\to \mathbb{Q}$ by 
$$\bone_C(c)=\begin{cases}
1 \quad \text{if } c\in C,\\
0 \quad \text{if } c\notin C.
\end{cases}$$
Clearly $\bone_C$ is constructible, and $CF(\X)$ is additively generated by characteristic functions. Define the function $\bone_\X: \X(\com)\to \mathbb{Q}$ to be $\bone_\X:=\bone_{\X(\com)}$. 

Let $f: \X\to \Y$ be a proper morphism of Deligne-Mumford stacks. In \cite{j}, Definition 4.17 (a), the notion of ``stack pushforward'' by $f$ is defined. This notion of pushforward, which we simply call {\em pushforward} and denote by $f_*$, will be used in this paper. We recall its definition as follows. Define a function $e_\X: \X(\com)\to \mathbb{Q}$ by $e_\X(c):= |G_c|$, where $|G_c|$ is the order of the isotropy group $G_c$ at the point $c\in \X(\com)$. A function $e_\Y: \Y(\com)\to \mathbb{Q}$ is similarly defined. Let $\phi: \X(\com)\to \mathbb{Q}$ be a constructible function. Define $f_*\phi: \Y(\com)\to \mathbb{Q}$ by 
\begin{equation}\label{pushforward}
f_*\phi(t):= e_\Y(t)\chi(\X(\com), \frac{1}{e_\X}\phi\bone_{f^{-1}(t)}), \quad t\in \Y(\com),
\end{equation}
where $\chi(-,-)$ is the weighted Euler characteristic as in \cite{j}, Definition 4.8. As pointed out in \cite{j}, page 599, since we work with Deligne-Mumford stacks, the pushforward $f_*$ is always defined. By \cite{j}, Corollary 4.14, the pushforard satisfies functoriality: $(f\circ g)_*= f_* g_*$.

\begin{lemma}\label{etale_map}
Let $f: \X\to \Y$ be a finite proper representable \'etale morphism of Deligne-Mumford stacks. Then the following equality of constructible functions hold:
$f_*\bone_\X=(\text{deg}\, f)\bone_{f(\X)}$.
\end{lemma}
\begin{proof}
For a geometric point $y\in f(\X(\com))$ we may write $f^{-1}(y)=\cup_{i\in I} x_i$, where $x_i\in \X(\com)$ and $I$ is a finite set. We have $f_*\bone_\X (y)=e_\Y(y)\sum_{i\in I} \frac{1}{e_\X(x_i)}$, which is equal to $\text{deg}\,f$.
\end{proof}

Consider again a Deligne-Mumford quotient stack $\X=[U/G]$ where $U$ is a smooth scheme and $G$ is a linear algebraic group. In \cite{o} the notion of constructible functions for this kind of stacks is also defined. By definition (see \cite{o}, Section 3.4) a constructible function on $\X$ is a $G$-invariant constructible function on $U$. Let $CF^G_{inv}(U)$ be the group\footnote{In \cite{o} this group is denoted by $\mathcal{F}^{G}_{inv}(U)$.} of $G$-invariant constructible functions on $U$. By \cite{o}, Lemma 3.3, this group is independent of the presentation of $\X$ as a quotient. Given a finite type substack $\Z\subset \X$, define $U_\Z:= \Z\times_\X U\subset U$. It is easy to see that the map $\bone_{\Z(\com)}\mapsto \bone_{U_\Z(\com)}$ defines an isomorphism 
\begin{equation}\label{CF_identification}
CF(\X)\simeq CF^G_{inv}(U). 
\end{equation}

It is also straightforward to check that under the identification above, the notion of pushforward for $CF^G_{inv}(U)$, as defined in \cite{o}, Section 2.6, coincides with the pushforward in (\ref{pushforward}). Let $f: \X\to \Y$ be a proper morphism of Deligne-Mumford stacks. Denote by $f_*: CF(\X)\to CF(\Y)$ the pushforward as in (\ref{pushforward}), and by $f_{*'}: CF(\X)\to CF(\Y)$ the pushforward in \cite{o}, Section 2.6 after the identification (\ref{CF_identification}). By construction we also have functorialty property $(f\circ g)_{*'}=f_{*'}g_{*'}$. Given $\phi\in CF(\X)$, to compare $f_*\phi$ and $f_{*'}\phi$ it suffices to compare them pointwise. Therefore we  may assume that $f$ is of the form $f: BG\to BH$ given by a homomorphism $G\to H$ of finite groups, and $\phi=\bone_{BG}$. 

In case $G$ is trivial, i.e. $f: \text{pt}\to BH$, we have $f_*\bone_{\text{pt}}=|H|\bone_{BH}$ by Lemma \ref{etale_map}. Let $H$ acts on itself by translations. Then we may present the map $f$ as a quotient by $H$ of the constant map $\tilde{f}: H\to \text{pt}$. Then it follows from the definitions that $f_{*'}\bone_{\text{pt}}=\tilde{f}_{*}\bone_{H}=|H| \bone_{\text{pt}}=|H| \bone_{BH}$. Thus $f_*=f_{*'}$ in this case. 

Suppose $G$ is not necessarily trivial. Let $p: \text{pt}\to BG$ be an atlas of $BG$. Then we have $p_*\bone_{\text{pt}}=|G|\bone_{BG}=p_{*'}\bone_{\text{pt}}$ and $(f\circ p)_*\bone_{\text{pt}}=|H|\bone_{BH}=(f\circ p)_{*'}\bone_{\text{pt}}$ by the case above. Thus by functoriality of pushforward, we have 
\begin{equation*}
f_*\bone_{BG}=\frac{1}{|G|}f_*p_*\bone_{\text{pt}}=\frac{1}{|G|}(f\circ p)_*\bone_{\text{pt}}=\frac{1}{|G|}(f\circ p)_{*'}\bone_{\text{pt}}=\frac{1}{|G|}f_{*'}p_{*'}\bone_{\text{pt}}=f_{*'}\bone_{BG}, 
\end{equation*}
which is what we want.

\section{Chern classes}
Let $\Delta: \X\to \X\times \X$ be the diagonal morphism. Recall that the {\em inertia stack} of $\X$ is defined to be $I\X:=\X\times_{\Delta,\X\times \X, \Delta}\X$, see for example \cite{dflnu}, Definition 5.1. and Lemma 5.2. Let $p:I\X\to \X$ be the natural projection.

If $\X=[M/G]$ where $M$ is a scheme and $G$ is a finite group, then the inertia stack can be described as follows: $$I[M/G]=\coprod_{(g): \text{conjugacy class of }G} [M^g/C_G(g)].$$
See for example \cite{dflnu}, Lemma 5.6 for a proof of this fact. 

The following is well-known.
\begin{lemma}[c.f. \cite{av}, Lemma 2.2.3]\label{local_cover}
Let $\X$ be a separated Deligne-Mumford stack, and $X$ its coarse moduli space. There is an \'etale covering $\coprod_a X_a\to X$ such that for each $a$ there is a scheme $U_a$ and a finite group $G_a$ acting on $U_a$, such that $\X\times_X X_a\simeq [U_a/G_a]$. 
\end{lemma}

\begin{proposition}\label{const_function_equality}
Let $\X$ be as in Assumption \ref{1st_assumption}. Then
$$\pi_*p_*\bone_{I\X}=\bone_X.$$
\end{proposition}
 
\begin{proof}
The question is local on $X$. Lemma \ref{etale_map} allows one to replace $X$ by an \'etale cover. In view of Lemma \ref{local_cover}, we are reduced to the case $\X=[M/G]$ where $M$ is a scheme and $G$ is a finite group. Denote by $\rho: M\to [M/G]$ the atlas map, and $\pi:[M/G]\to M/G$ the map to coarse moduli scheme. 

Put $\alpha:=\frac{1}{|G|}\sum_{g\in G}\bone_{M^g}$. For a geometric point $x\in M$ we denote by $[x]$ the corresponding geometric point in $M/G$. The calculation in the proof of \cite{o}, Proposition 6.1 gives 
\begin{equation*}
\begin{split}
(\pi\circ\rho)_*\alpha([x])
&=\frac{1}{|G|}\sum_{g\in G}(\pi\circ\rho)_*\bone_{M^g}([x])\\
&=\frac{1}{|G|}\sum_{x'\in G.x}\sum_{g\in G}\bone_{M^g}(x')\\
&=\frac{1}{|G|}|G.x||\text{Stab}_x(G)|=1.
\end{split}
\end{equation*}
Thus $(\pi\circ\rho)_*\alpha=\bone_{M/G}$. It is easy to see that $$\frac{1}{|G|}\sum_{g\in G}\bone_{M^g}=\sum_{(g): \text{conjugacy class}}\frac{\bone_{M^g}}{|C_G(g)|}.$$
By Lemma \ref{etale_map}, we see that the pushforward of $\bone_{M^g}$ to $[M^g/C_G(g)]$ is equal to $|C_G(g)|\bone_{[M^g/C_G(g)]}$. Therefore the pushforward of $\alpha$ to $I[M/G]$ is equal to $\bone_{I[M/G]}$. The result follows.
\end{proof}

\begin{remark}
In the proof of Proposition \ref{const_function_equality} one can argue without using Lemmas \ref{local_cover} and \ref{etale_map}, as follows: Let $x: \text{Spec}\, k\to X$ be a geometric point. Then $\text{Spec}\, k\times_{x,X, \pi}\X$ is isomorphic to $BG$ for some finite group $G$. Moreover, we have 
$$\text{Spec}\, k\times_{x, X, \pi\circ p}I\X\simeq \coprod_{(g): \text{conjugacy class of }G} [\text{Spec}\, k/C_G(g)].$$
We conclude by using the equality $$\sum_{(g): \text{conjugacy class of }G}\frac{1}{|C_G(g)|}=1.$$
\end{remark}

\begin{theorem}\label{csm_class}
Let $\X$ be as in Assumption \ref{1st_assumption}. Then
$$\pi_*p_*c(T_{I\X})=c^{SM}(X).$$
\end{theorem} 
 
\begin{proof}
Assumption \ref{1st_assumption} on $\X$ implies that $I\X$ is a quotient stack $[W/H]$ of a quasi-projecitve scheme $W$ by a linear algebraic group $H$ (see \cite{k}, Theorem 4.4). This allows us to apply \cite{o}, Theorem 3.5 to Proposition \ref{const_function_equality} to obtain $$\pi_*p_*C_*(\bone_{I\X})=c^{SM}(X)\cap [X].$$

The function $\bone_{I\X}$ is identified with $\bone_W$ under the identification $CF(I\X)\simeq CF^H_{inv}(W)$ of groups of constructible functions. This implies that $C_*(\bone_{I\X})=C_*^H(\bone_W)$. 
The normalization property of $C_*^H$ implies that $C_*^H(\bone_W)=c^H(T_W)\cap [W]_H$. Since the $H$-equivariant Chern class $c^H(T_W)$ is identified with the Chern class $c(T_{I\X})$ under the identification $H_H^*(W)\simeq H^{*}(I\X)$, the result follows.
\end{proof}
 
\section{Stringy Chern classes} 
Let $IX$ be the coarse moduli scheme of the inertia stack $I\X$. There is a diagram
 $$\begin{CD}
I\X@> p>>  \X \\
@V{\tpi}VV @V{\pi}VV \\
IX@>\pbar>> X.
\end{CD}$$

In \cite{dflnu} the authors define a constructible function $\Phi_X$ and the {\em stringy Chern class} of $X$: $$c_{str}(X):=c^{SM}(\Phi_X).$$



\begin{theorem}\label{stringy_class}
Let $\X$ be as in Assumption \ref{1st_assumption}. Then
\hfill
\begin{enumerate}
\item
$\Phi_X=\pbar_*\bone_{IX}$.
\item
$c_{str}(X)=\pi_*q_*c(T_{II\X})$, where $q: II\X\to \X$ is the natural projection from the {\em double inertia stack} $II\X$ to $\X$.
\end{enumerate}
\end{theorem}
\begin{proof}
Let $g: W\to V$ be a resolution of singularity, $f: V'\to V$ an \'etale map, $W':=W\times_V V'$, and $f': W'\to W$, $g': W'\to V'$ the natural projections. Then $g'$ is also a resolution of singularity, and we have 
\begin{equation*}
\begin{split}
f'^*K_{W/V}&=f'^*K_W-f'^*g^*K_V\\
&= K_{W'}-g'^*f^*K_V \quad (\text{since }f' \text{ is \'etale})\\
&=K_{W'}-g'^*K_{V'},
\end{split}
\end{equation*}
where the last step is justified as follows: Let $j: U\to V$ be the smooth locus of $V$, $U':=U\times_V V'$, and $j': U'\to V'$, $f_U: U'\to U$ the natural projections. Then we have $K_V:=j_*(\wedge^{\text{dim}\,U}\Omega_U^1)$. Thus
\begin{equation*}
f^*K_V=f^*j_*(\wedge^{\text{dim}\,U}\Omega_U^1)=j'_*f_U^*(\wedge^{\text{dim}\,U}\Omega_U^1)=j'_*(\wedge^{\text{dim}\,U'}\Omega_{U'}^1)=K_{V'}.
\end{equation*}
It follows that $f'^*K_{W/V}=K_{W'/V'}$. Applying \cite{dflnu}, Proposition 2.3, we see that part (1) can be checked on an \'etale covering of $X$. Therefore by Lemma \ref{local_cover} we may assume that $X=M/G$ for some scheme $M$ and finite group $G$. In this case part (1) is \cite{dflnu}, Theorem 6.1. 

Since the inertia stack of $I\X$ is by definition the double inertia stack $II\X$, part (2) follows immediately from Theorem \ref{csm_class} applied to $IX$.
\end{proof}

Let $e_{str}(X)$ be the stringy Euler characteristic of $X$. By \cite{dflnu}, Proposition 4.4, we have $e_{str}(X)=\int_Xc_{str}(X)$. The following is immediate from Theorem \ref{stringy_class}.

\begin{corollary}\label{euler_char}
$e_{str}(X)=\int_{II\X}c_{top}(T_{II\X})$.
\end{corollary}

Let $n=\text{dim}\, X$. In \cite{b}, Definition 3.1, Batyrev defined a number $c_{st}^{1,n-1}(X)$, which can be interpreted as a stringy analogue of the Chern number $c_1(X)c_{top-1}(X)$. Properties of $c_{str}(X)$ (see the proof of \cite{dflnu}, Proposition 4.4) implies that
$$c_{st}^{1, n-1}(X)=\int_X c_1(X) c_{str}(X).$$

Theorem \ref{stringy_class} implies 

\begin{corollary}
$$c_{st}^{1,n-1}(X)=\int_{II\X}q^*\pi^*c_1(X)c_{top-1}(T_{II\X}).$$
\end{corollary}
\begin{remark}
This proves a more general form of \cite{jt}, Conjecture A.2.
\end{remark}

Under additional hypotheses we can deduce some consequences on orbifold Hodge numbers.

\begin{assump}\label{additional_assump}
Let $\X$ be as in Assumption \ref{1st_assumption}. In addition $X$ is assumed to be Gorenstein, the map $\pi: \X\to X$ is assumed to be birational, and $K_\X=\pi^*K_X$.
\end{assump}

Let $I\X=\coprod_{i\in\sI}\X_i$ be the decomposition into disjoint union of connected components. For each $\X_i$ one can associate a rational number $age(\X_i)$ called the {\em age} of $\X_i$. See for example \cite{ya} for a definition.

The numbers $age(\X_i)$, which arise naturally in the context of Riemann-Roch formula for twisted curves (see \cite{agv}, Section 7.2), are relevant to us due to their presence in the {\em Chen-Ruan orbifold cohomology}. By definition, the Chen-Ruan cohomology groups of $\X$ are $H_{CR}^*(\X,\com):=H^*(I\X, \com)=\oplus_{i\in \sI}H^*(\X_i, \com)$. The numbers $age(\X_i)$ are used to define a new grading on $H_{CR}^*(\X, \com)$:
$$H_{CR}^p(\X,\com):=\oplus_{i\in \sI} H^{p-2age(\X_i)}(\X_i, \com).$$
The Dolbeault cohomology version of this can be similarly defined: 
$$H_{CR}^{p,q}(\X, \com):=\oplus_{i\in \sI} H^{p-age(\X_i), q-age(\X_i)}(\X_i, \com).$$

\begin{proposition}\label{orb_numbers}
Let $\X$ be as in Assumption \ref{additional_assump}. Then the following holds.
\begin{equation}\label{c_st_1n-1}
c_{st}^{1,n-1}(X)=\int_{II\X}q^*c_1(\X)c_{top-1}(T_{II\X}).
\end{equation}
\begin{equation}\label{trace}
\sum_{i\in \sI}\sum_{p\geq 0}(-1)^p\left(p+age(\X_i)-\frac{\text{dim}\X}{2}\right)^2\chi(\X_i,\Omega_{\X_i}^p)=\frac{1}{12}\int_{II\X}\text{dim}\X c_{top}(T_{II\X})+2c_1(T_\X)c_{top-1}(T_{II\X}).
\end{equation}

\end{proposition}

\begin{proof}
(\ref{c_st_1n-1}) follows from $K_\X=\pi^*K_X$.

We now prove (\ref{trace}). Under Assumption \ref{additional_assump} a result of T. Yasuda \cite{ya} asserts that Batyrev's {\em stringy Hodge numbers} $h_{st}^{p,q}(X)$ coincide with {\em orbifold Hodge numbers} $h_{orb}^{p,q}(\X):=\text{dim}H_{CR}^{p,q}(\X, \com)$. We refer to \cite{ya} for relevant definitions. In terms of generating functions, we have $$E_{st}(X, s,t)=E_{orb}(\X,s,t),$$ where $$E_{st}(X,s,t):=\sum_{p, q\geq 0}(-1)^{p+q}h_{st}^{p,q}(X)s^pt^q$$ is the stringy E-polynomial and $$E_{orb}(\X,s,t):=\sum_{p, q\geq 0}(-1)^{p+q}h_{orb}^{p,q}(\X)s^pt^q$$ is the orbifold E-polynomial. Combining this with Corollary 3.10 of \cite{b} we find that
\begin{equation}
\sum_{p, q}(-1)^{p+q}\left(p-\frac{\text{dim}\X}{2}\right)^2h_{orb}^{p,q}(\X)=\frac{\text{dim}X}{12}e_{str}(X)+\frac{1}{6}c_{st}^{1,n-1}(X).
\end{equation}
We rewrite the left-hand side as follows. By definition of $h_{orb}^{p,q}(\X)$, 
\begin{equation*}
\begin{split}
&\sum_{p, q}(-1)^{p+q}\left(p-\frac{\text{dim}\X}{2}\right)^2h_{orb}^{p,q}(\X)\\
=&\sum_{i\in\sI} \sum_{p, q}(-1)^{p+q}\left(p-\frac{\text{dim}\X}{2}\right)^2\text{dim}H^{p-age(\X_i), q-age(\X_i)}(\X_i, \com)\\
=&\sum_{i\in\sI} \sum_{p, q\geq 0}(-1)^{p+q+2age(\X_i)}\left(p+age(\X_i)-\frac{\text{dim}\X}{2}\right)^2\text{dim}H^{p, q}(\X_i, \com) \quad (\text{re-indexing})\\
=&\sum_{i\in\sI} \sum_{p\geq 0}(-1)^p\left(p+age(\X_i)-\frac{\text{dim}\X}{2}\right)^2\left(\sum_{q\geq 0}(-1)^q\text{dim}H^{p, q}(\X_i, \com)\right). 
\end{split}
\end{equation*}
In the last equality we used the fact that $age(\X_i)\in \mathbb{Z}$, which is true because $X$ is Gorenstein. Thus we arrive at 
\begin{equation}\label{evidence}
\sum_{i\in \sI}\sum_{p\geq 0}(-1)^p\left(p+age(\X_i)-\frac{\text{dim}\X}{2}\right)^2\chi(\X_i,\Omega_{\X_i}^p)=\frac{\text{dim}X}{12}e_{str}(X)+\frac{1}{6}c_{st}^{1,n-1}(X).
\end{equation}
(\ref{trace}) now follows from Corollary \ref{euler_char} and (\ref{c_st_1n-1}).
\end{proof}

\begin{remark}
(\ref{trace}) is conjectured to hold for any smooth proper Deligne-Mumford stack with projective coarse moduli space, see \cite{jt}, Conjecture 3.2'.
\end{remark}


\begin{thebibliography}{12}
\bibitem{agv} D. Abramovich, T. Graber, and A. Vistoli, Gromov-Witten theory of Deligne-Mumford stacks,  {\em Amer. J. Math.} 130 (2008) no. 5, 1337--1398.

\bibitem{av} D. Abramovich and A. Vistoli, Compactifying the space of stable maps, {\em J. Amer. Math. Soc.} 15 (2002), no. 1, 27--75.

\bibitem{alu} P. Aluffi, Celestial integration, stringy invariants, and Chern-Schwartz-MacPherson classes, in {\em Real and complex singularities}, 1--13, Trends Math., Birkh\"auser, Basel, 2007.

\bibitem{b} V. Batyrev, Stringy Hodge numbers and Virasoro algebra, {\em Math. Res. Lett.} 7 (2000), no. 2-3, 155--164.

\bibitem{dflnu}  T. de Fernex, E. Lupercio, T. Nevins, and B. Uribe, Stringy Chern classes of singular varieties, {\em Advances in Math.} 208 (2007), 597-621, arXiv:math/0407314.

\bibitem{eg} D. Edidin, W. Graham, Equivariant intersection theory,  {\em Invent. Math.}  131  (1998),  no. 3, 595--634. 

\bibitem{jt} Y. Jiang and H.-H. Tseng, On Virasoro Constraints for Orbifold Gromov-Witten Theory,  arXiv:0704.2009.

\bibitem{j} D. Joyce, Constructible functions on Artin stacks, {\em J. London Math. Soc.} (2) 74 (2006), no. 3, 583--606.

\bibitem{k} A. Kresch, On the geometry of Deligne-Mumford stacks, to appear in {\em Algebraic Geometry (Seattle 2005)}, Proc. Symp. Pure Math., Vol. 80, Amer. Math. Soc. 2009.

\bibitem{o} T. Ohmoto, Equivariant Chern classes of singular algebraic varieties with group actions,  {\em Math. Proc. Cambridge Philos. Soc.} 140 (2006), no. 1, 115--134. 

\bibitem{vis} A. Vistoli, Intersection theory on algebraic stacks and on their moduli spaces, {\em Invent. Math.}  97  (1989),  no. 3, 613--670.

\bibitem{ya} T. Yasuda, Motivic integration over Deligne-Mumford stacks, {\em Advances in Math.} 207 (2006), 707--761.
\end{thebibliography}
\end{document}